\documentclass[12pt, oneside]{amsart}

\usepackage{amsmath,amssymb}
\usepackage{accents}
\usepackage{graphicx}
\usepackage{verbatim}
\usepackage{color}
\usepackage{float}
\usepackage{enumerate}
\usepackage[textwidth=6in, textheight=9in,marginpar=0.75in]{geometry}

\newlength{\dhatheight}

\newtheorem{theorem}{Theorem}[section]
\newtheorem{conjecture}[theorem]{Conjecture}
\newtheorem{proposition}[theorem]{Proposition}
\newtheorem{lemma}[theorem]{Lemma}

\newtheorem{corollary}[theorem]{Corollary}

\newcommand{\nidt}{\noindent}

\newcommand{\bjset}{{\sf BJSet}}
\newcommand{\bjref}{{\sf RefBJSet}}

\begin{document}
\title[A counterexample to the Bernhard-Jablan Conjecture]
{A counterexample to the Bernhard-Jablan Unknotting Conjecture}

\author[M.~Brittenham]{Mark Brittenham}
\address{Department of Mathematics\\
        University of Nebraska\\
         Lincoln NE 68588-0130, USA}
\email{mbrittenham2@math.unl.edu}

\author[S. Hermiller]{Susan Hermiller}
\address{Department of Mathematics\\
        University of Nebraska\\
         Lincoln NE 68588-0130, USA}
\email{hermiller@unl.edu}

\date{July 18, 2017}

\begin{abstract}
We show that there is a knot 
satisfying the property that for each minimal
crossing number diagram of the knot and each single crossing of
the diagram, 
changing the crossing results in a diagram for a knot whose 
unknotting number is at least that of the original knot,
thus giving 
a counterexample to the Bernhard-Jablan Conjecture.
\end{abstract}

\maketitle


\section{Introduction}\label{sec:intro}


The unknotting number $u(K)$ of a knot $K$ is one of the most fundamental
measures of the complexity of a knot. It is defined as the 
minimum number of crossing changes, interspersed with isotopy,
required to transform a diagram of $K$ to a diagram of the unknot. A
wide array of techniques have been employed over the years
to compute unknotting numbers, using nearly every new technique
that has been introduced into knot theory
(see 
\cite{coli},\cite{kawa},\cite{krmr},\cite{li02},\cite{mc17},\cite{mu65},\cite{ow08},\cite{ow10},\cite{ozsz},\cite{ru93} 
for a selection). Yet to this day there
are still nine 10-crossing knots whose unknotting numbers remain unknown,
a testament to the difficulty of determining this basic invariant.
See the Knotinfo site \cite{knotinfo} for the most up-to-date list of unresolved knots.

It is an elementary fact that for every diagram $D$ of $K$ with $n$
crossings, there is a set of $k\leq n/2$ crossings in the diagram which 
when all $k$ are reversed yields a diagram $D^\prime$ of the unknot. So every diagram 
has its own `unknotting number'.
It is also a well-known result that the unknotting number of $K$ can be 
defined, alternatively, as the minimum,
over all diagrams $D$ of the knot $K$, of the unknotting number of the 
diagram. That is, one can always arrange things so that, in a minimal unknotting
sequence, all isotopies come first. It is 
natural then to try to turn this alternate formulation into an algorithm 
to compute the unknotting number of a knot $K$, by trying to 
limit the number or types of diagrams of $K$ that need to be considered
(either with or without intermediate isotopy).

The most natural initial conjecture, namely that we can limit ourselves to the (finitely many)
diagrams of $K$ with minimum crossing number,
without isotopy, was disproved by Bleiler \cite{bleiler} and Nakanishi \cite{nakan83},
who showed that the unique minimal diagram of the knot $10_8$ requires
three crossing changes to produce the unknot, while $u(10_8)=2$.
There is a crossing change in the diagram which reduces unknotting 
number, but an isotopy of the resulting knot is required to create
the next (and last) needed crossing change.

The next best thing one could hope for is that one of the 
minimal diagrams for $K$ admits a crossing change which lowers
the unknotting number (as is true for the knot $10_8$). 
This was the conjecture posited by
Bernhard \cite{bernhard} and Jablan \cite{jablan98}, and which has come to be known
as the Bernhard-Jablan Conjecture.
For a knot $K$, we let
\begin{align*}
\bjset_K = \{K' \mid &  K^\prime \textrm{ is obtained from a 
minimal crossing number}\\
&\textrm{ projection of } K \textrm{ by changing a single crossing}\},
\end{align*}
and we write
$u_{BJ}^s(K)$ to denote 
the {\it strong Bernhard-Jablan unknotting number} of $K$, 
defined as
\begin{equation*}
u_{BJ}^s(K) = 1 + \min\{u(K^\prime) \mid  K^\prime \in \bjset_K\}.
\end{equation*}

\begin{conjecture}[Bernhard-Jablan] \label{conj:bj}
Every knot $K$ possesses a minimal-crossing-number diagram $D$ and a 
crossing in $D$, such that changing the crossing results in a diagram
$D^\prime$ for a knot $K^\prime$ with $u(K^\prime)<u(K)$.  Equivalently,
\begin{equation}\label{eqn:bj}
u(K) = u_{BJ}^s(K)
\end{equation} 
for all knots $K$.
\end{conjecture}

If this conjecture were true, then we would have, in principle, 
an algorithm to compute $u(K)$: construct the finitely many
minimal diagrams for $K$, change each crossing, and compute 
the unknotting number of each resulting knot (recursively, using
the conjecture). The minimum unknotting number found among 
the resulting knots is one
less (since under crossing change, unknotting number can change by at most
one) than the unknotting number of $K$. Unfortunately, this cannot work
in general; that is the main result of this paper.

\begin{theorem} \label{thm:fail}
There is a knot $K$ whose unknotting number is less than or equal to 
the unknotting numbers of all of the knots obtained
by changing any single crossing in each of the minimal crossing diagrams of 
$K$. That is, the Bernhard-Jablan Conjecture is false, in general.
\end{theorem}

It is an interesting twist that while we now know that the
conjecture is false, we do not (yet) know which knot $K$
fails to satisfy Equation~(\ref{eqn:bj}), and consequently
satisfies $u(K) < u_{BJ}^s(K)$.  To make this more precise,
we define the {\it weak Bernhard-Jablan unknotting number} $u_{BJ}^w(K)$
of any knot $K$ by induction on the crossing number, as follows.
Let $u_{BJ}^w(\textrm{unknot})=0$, and suppose that we have computed 
$u_{BJ}^w(K)$ for all knots with
crossing number $< n$. 
Let $S=S_n$ be the set of knots with crossing number $n$,
and for each $K \in S$, let $T_K$
be the set obtained from the set $\bjset_K$
(of knots represented by a diagram that can be obtained
from a minimal crossing diagram for $K$ by changing a single crossing)
by replacing each knot $K' \in T_K$ with crossing number $< n$
 with the number $u_{BJ}^w(K')$ (using the inductive assumption).
We make repeated passes through the set $S$ and the $T_K$ sets, as follows. 
For every knot $K$ in $S$ such that $T_K$ contains the number 0 (that is,
$\bjset_K$ contains the unknot), define
$u_{BJ}^w(K)=1$ (since $K$ is not the unknot), remove $K$ from $S$, and 
for any knots $K'$ remaining in $S$, replace every instance of $K$
in $T_{K'}$ by the number 1.
For each natural number $\ell$ starting at 1, incremented by 1 at each repetition,
we iterate this process:
For all $K \in S$ such that $\ell \in T_K$, define
$u_{BJ}^w(K)=\ell+1$, remove $K$ from $S$, and 
for any knots $K'$ remaining in $S$, replace every instance of $K$
in $T_{K'}$ by the number $\ell+1$.

In Lemma~\ref{lem:bjnumbers}, we show that $u_{BJ}^w(K)$ is
(well-)defined, 
that is, the process just described will assign a value to $u_{BJ}^w(K)$,
and $u(K) \le u_{BJ}^s(K) \le u_{BJ}^w(K)$ for all knots $K$.
In Corollary~\ref{cor:bjequiv}, we show that the Bernhard-Jablan
conjecture holds if and only if 
\begin{equation}\label{eqn:weakbj}
u(K) = u_{BJ}^w(K)
\end{equation}
for all knots $K$.
Our proof of Theorem~\ref{thm:fail} shows that Equation~(\ref{eqn:weakbj})
fails for the knot $K13n3370$, and shows that Equation~(\ref{eqn:bj}) fails
for one of four knots.  More precisely, in Section~\ref{sec:dat1} we show:

\begin{theorem} \label{thm:example}
The knots $K12n288$, $K12n491$, $K12n501$, and $K13n3370$ satisfy
$K12n288, K12n491, K12n501 \in \bjset_{K13n3370}$, and the following hold.
\begin{enumerate}[(a)]
\item $u(K13n3370)\leq 2 < 3=u_{BJ}^w(K13n3370)$.
\item  $u_{BJ}^s(K')=u_{BJ}^w(K')=2$ for all $K' \in \{K12n288, K12n491, K12n501\}$.
\item $u_{BJ}^s(K13n3370)=1+\min\{u(K') \mid K' \in  \{K12n288, K12n491, K12n501\}\}$.
\item For at least one $K\in\{K12n288,K12n491,K12n501,K13n3370\}$ we have
$u(K)<u_{BJ}^s(K)$. 
\end{enumerate}
\end{theorem}

The proof that the unknotting number of 
the knot $K13n3370$ is less than or equal to $2$ is given in Section~\ref{sec:dat1}
by direct construction.
Part (d) of Theorem~\ref{thm:example}
is a direct corollary of parts (a-c):
If Equation~(\ref{eqn:bj}) holds for
the three 12-crossing knots, then by part (b) and Lemma~\ref{lem:bjnumbers}(2) they all
have unknotting number equal to 2, and so by part (c) the strong
Bernhard-Jablan unknotting number for $K13n3370$ is 3, and then
part (a) shows that this 13-crossing knot fails Equation~(\ref{eqn:bj}). 

The bulk of the work needed to reach these conclusions
was carried out by computer. In particular, the diagram which
establishes that $u(K13n3370)\leq 2$ was found by a random
search, using the program SnapPy \cite{snappy} to generate and identify
knots as well as identify the knots obtained by crossing change.
That search found a projection of the knot $K13n3370$ which a 
single crossing change turned into a diagram for the knot 
$K11n21$, which has unknotting number one. [To be completely accurate,
it found the diagram for $K11n21$ first.]
The results on weak and strong Bernhard-Jablan unknotting
numbers are established using an exhaustive search 
to identify all of the minimal crossing
diagrams of the knots involved, as an incidental consequence
of identifying all minimal crossing diagrams of 
all knots through 14 crossings. This last computation was also
carried out within SnapPy, and later independently
verified using the program Knotscape \cite{knotscape}. The
code used to carry out these computations can be found on the authors' website,
at the URL listed on page \pageref{url} below.

In Section~\ref{sec:further}, we provide information on further
examples of knots for which the unknotting and weak Bernhard-Jablan
unknotting numbers differ.  In Section~\ref{sec:fut}, we discuss
several open questions that arise from Theorem~\ref{thm:fail}.

All of these computations were in fact part of a larger project, whose goal
is to fill in the gaps in our knowledge of the unknotting numbers
of low crossing-number knots. By using the fact that unknotting
number changes by at most one under crossing change, we can
use the knowledge of the unknotting numbers of `crossing-adjacent'
knots to pull a lower bound $L$ for $u(K)$ up (by finding 
an adjacent knot $K^\prime$ with $u(K^\prime)> L+1$) or pull an upper
bound $U$ down (by finding adjacent $K^\prime$
with $u(K^\prime)< U-1$). Pulling lower bounds up is a relatively 
routine occurrence in our computations (and often results in the
determination of the unknotting number); the first instance where an
upper bound was pulled down forms the core of this paper.

\subsection*{Acknowledgments}

The authors wish to thank the Holland Computing Center at the 
University of Nebraska, which provided the computing
facilities on which the bulk of this work was carried out.
The second author acknowledges support by NSF grant DMS-1313559.


\section{The proof of Theorem~\ref{thm:fail}} \label{sec:dat1}



\subsection{The three unknotting numbers} \label{sub:bjlemmas}


We begin the proof of Theorem~\ref{thm:fail} with two lemmas 
on the relationships between the three types of unknotting numbers.

\begin{lemma}\label{lem:bjnumbers}
Let $K$ be a knot.
\begin{enumerate}
\item The weak Bernhard-Jablan unknotting number is well-defined and satisfies
\begin{equation*}
u_{BJ}^w(K) = 1 + \min\{u_{BJ}^w(K^\prime) \mid  K^\prime \in \bjset_K\}.
\end{equation*}
\item $u(K) \le u_{BJ}^s(K) \le u_{BJ}^w(K)$.
\end{enumerate}
\end{lemma}

\begin{proof}
In order to show that $u_{BJ}^w(K)$ is well-defined, we must show that
in the iterative procedure passing through the set $S$ and the
sets $T_K$ (for $K \in S$),  eventually $S = \emptyset$ and
$u_{BJ}^w(K)$ is defined for every knot with crossing number $n$.
Suppose that $K \in S$, let $D$ be any
minimal crossing number diagram for $K$, and let $T_{D,k}$ be the
set of knots represented by diagrams
obtained by changing any $k$ crossings of the diagram $D$.
By the elementary fact noted above (bounding the number of crossing
changes needed in a minimal crossing diagram to produce the unknot),
for some $1 \le k \le n/2$, the
sets $T_{D,0},...,T_{D,k-1}$ do not contain a knot with crossing number $< n$,
but the set $T_{D,k}$ does contain a knot $K_k$
with crossing number $< n$.  Now $K_k$ is represented by
a diagram obtained from $D$ by changing $k$ crossings $c_1,...,c_k$.  
Moreover, the knots $K_i$ represented by the diagram obtained
from $D$ by changing the crossings $c_1,...,c_i$ satisfy
that $K_i$ has crossing number $n$ and
$K_{i+1} \in \bjset_{K_i}$ 
for all $i < k$, and $K=K_0$.  Now the set $T_{K_{k-1}}$ contains the number
$u_{BJ}^w(K_k)$, and so in the procedure above, after at most  
$1+u_{BJ}^w(K_k)$ passes
through $S$ and the $T_{\tilde K}$ sets, 
$u_{BJ}^w(K_{k-1})$ is assigned a value $\leq 1+u_{BJ}^w(K_k)$.
Then after at most 
$1+u_{BJ}^w(K_{k-1})$ passes,
$u_{BJ}^w(K_{k-2})$ is assigned a value $\leq 1+u_{BJ}^w(K_{k-1})$, etc.
Hence after finitely many passes, $u_{BJ}^w(K)=u_{BJ}^w(K_0)$
is assigned a value.

From the definition of $u_{BJ}^w(K)$, 
at the point in the process in which $u_{BJ}^w(K)=r$ is defined
each knot $K'$ in the set $\bjset_K$ is either in $T_K$ or
else has been replaced in $T_K$ by the number $u_{BJ}^w(K')$.  
Moreover, at that point $r-1$ is the least number appearing in $T_K$,
and all of the knots $K'$ in $T_K$ satisfy $u_{BJ}^w(K') \ge r$
as well, 
since they will be assigned values later in the process,
concluding the proof of part (1).

The inequality $u(K) \le u_{BJ}^s(K)$ is immediate from the definitions, 
since for every $K^\prime\in\bjset_K$, we have $|u(K^\prime)-u(K)|\leq 1$, so 
$u(K^\prime)\geq u(K)-1$.
An inductive argument shows that $u(K)\leq u_{BJ}^w(K)$,
since $u_{BJ}^w(K)=u_{BJ}^w(K^\prime)+1\geq u(K^\prime)+1\geq u(K)$ for 
some $K^\prime\in\bjset_K$ with $u_{BK}^w(K^\prime)<u_{BJ}^w(K)$,
and therefore (using part (1)) $ u_{BJ}^s(K)\leq  u_{BJ}^w(K)$ as well.
\end{proof}

\begin{corollary}\label{cor:bjequiv}
The Bernhard-Jablan Conjecture holds if and only if 
$u(K)=u_{BJ}^w(K)$
for all knots $K$.
\end{corollary}

\begin{proof}
If $u(K)=u_{BJ}^w(K)$ (Equation~(\ref{eqn:weakbj}))
holds for all knots $K$, then Lemma~\ref{lem:bjnumbers}(2)
shows that Equation~(\ref{eqn:bj}) holds for all $K$ as well.

Conversely, suppose that the Bernhard-Jablan Conjecture holds.  We prove
Equation~(\ref{eqn:weakbj}) by induction on $u(K)$.
Note that $u(K)=0$ implies that $K$ is the unknot, and so $u_{BJ}^w(K)=0$ also.
Suppose that Equation~(\ref{eqn:weakbj}) holds for all knots with
unknotting number $< n$, and suppose that $u(K)=n$.
By Equation~(\ref{eqn:bj}) and the definition of $u_{BJ}^s$, 
$u(K)=1+\min\{u(K') \mid K' \in \bjset_K\}$,
and so by the inductive assumption
$u(K)=1+\min\{u_{BJ}^w(K') \mid K' \in \bjset_K\}$,
since the smallest $u(K^\prime)$ appearing is less than $u(K)=n$, and so equals
$u_{BJ}^w(K^\prime)$, and for all other $K^{\prime\prime}\in \bjset_K$ we have
$u_{BJ}^w(K^{\prime\prime})\geq u(K^{\prime\prime})\geq u(K^\prime)=u_{BJ}^w(K^\prime)$.
The conclusion follows from Lemma~\ref{lem:bjnumbers}(1).
\end{proof}


\subsection{Computational resources} \label{sub:software}

 
There are three resources that we used to assemble the data needed 
to establish the results in Theorem~\ref{thm:example}. The first of these 
is {\it SnapPy} \cite{snappy}, 
an indispensible program for studying 
knots and 3-manifolds,
based on the program SnapPea, and developed by 
Culler, Dunfield, Goerner and Weeks. In particular, SnapPy
was used to build knots as (random) braids, and to identify them
using SnapPy's {\it identify()} command. This utility uses a hash based on the knot's
geometric properties to compare to a list, to determine the likely identity of the knot, 
and then uses an underlying
canonical triangulation of the knot complement to find a combinatorial
isomorphism between ideal triangulations of a given knot complement and 
a reference knot complement, which provides a homeomorphism
of knot complements and so, by the solution to the Knot Complement Problem
by Gordon and Luecke \cite{golu}, identifies the knot. We also detected the
unknot using the {\it fundamental\_group()} command, since the unknot is the
only knot whose knot group has a presentation with no relators,
i.e., the knot group is free. Although it is not relevant to the specific
examples that we will discuss, SnapPy's {\it deconnect\_sum()} command was
also used to identify non-prime knots and their summands.

The second resource we used is Knotscape \cite{knotscape}, developed by Hoste and 
Thistlethwaite. We used this program to provide an independent 
verification of the database of minimal crossing diagrams for knots with
a fixed crossing number, discussed below; Knotscape works directly 
with Dowker-Thistlethwaite (DT) codes, and
applies combinatorial moves on these codes to match a knot projection with 
a unique normal form, thereby identifying the knot. It also, incidentally,
can be much faster than SnapPy at the task of identifying thousands or millions
of knots in a row, since it directly manipulates a knot's DT code. 

The third resource we used is {\it Knotinfo} \cite{knotinfo}, the online database of 
knot invariants maintained by Jae Choon Cha and Chuck Livingston at 
the University of Indiana.
This was used to populate lists of knots with known unknotting
number, as well as initial upper and lower bounds for knots with
unknown unknotting number, in order to make the needed comparisons
with `crossing-adjacent' knots.

Taken together, these tools gave us the foundation on which to build 
our search process, and assemble the data described below.


\subsection{Computation of $\bjset_K$} \label{sub:bjset}

 
In this subsection we describe our procedure for constructing
the set $\bjset_K$ for any $k$-crossing knot $K$.
Roughly, this is done
by exhaustively identifying every
knot associated to every $k$-crossing diagram up to flypes (Figure~\ref{fig:flype1}). 
We carried out these calculations using SnapPy to identify the knots,
and, as an independent check, carried out the same exhaustion
using Knotscape \cite{knotscape}, again using SnapPy to make the identifications. 

More precisely, for a fixed natural number $k$,
we start with a set of DT codes (as used in SnapPy) containing a representative 
for a minimal diagram of each alternating $k$-crossing
knot; we refer to the associated set of knot diagrams as the 
set $\mathcal{R}$ of {\it reference diagrams}
for this procedure.

Let $\mathcal{D}$ be the set of all DT codes for diagrams obtained by 
changing a subset of the crossings of a reference diagram
(omitting changing the lexicographically first crossing in
each DT code, since including it would only produce the mirror images
of the other knots built).  SnapPy can then be used to identify the knots
associated to the DT codes in $\mathcal{D}$;
let $\mathcal{D_K}$ be the subset of $\mathcal{D}$ consisting of the 
projections of $K$ obtained from the reference diagrams by crossing changes.
Finally, again use SnapPy to identify the set $\bjref_K$ 
of knots represented by diagrams obtained from elements of $\mathcal{D}_K$
by a single crossing change.  

In the following, we show that the collection $\bjset_K$ of knots we obtain from all $k$-crossing
projections of a $k$-crossing knot $K$ by changing a single crossing is 
the same as the collection $\bjref_K$ we obtain by using the projections of 
$K$ obtained from crossing changes made to reference diagrams alone;
hence the algorithm above computes the set $\bjset_K$,
and the diagrams obtained from the reference diagrams
by a single crossing change 
suffice for the computations needed in the strong and weak Bernhard-Jablan unknotting numbers.
This perspective is implicit in some of the work of Jablan and Sazdanovic~\cite{jablan2}
and of Zekovi\'c, Jablan, Kauffman, Sazdanovic and Sto\v{s}i\'c~\cite{zjkss}
on this conjecture.

\begin{proposition}\label{prop:refsuffices}
For any knot $K$, $\bjset_K = \bjref_K$.
\end{proposition}

\begin{proof}
It is immediate that $\bjref_K \subseteq \bjset_K$.

Let $K_1$ be any element of $\bjset_K$.  There is a
$k$-crossing projection $D$ of the $k$-crossing knot $K$, and a diagram 
$D_1$
obtained from $D$ by changing a single crossing, such that $D_1$ is a projection of $K_1$.
Note that $D$ is reduced (i.e., $D$ has no nugatory crossing)
since such a $k$-crossing diagram with a nugatory crossing
cannot represent a $k$-crossing knot.
Build the  associated knot shadow (a $k$-vertex
4-valent graph) by replacing each crossing with a vertex
of valence 4, and from this shadow, reintroduce crossings
to obtain a reduced diagram $D^\prime$ for an alternating knot.
The diagram $D^\prime$ 
must have crossing number $k$~\cite{kauf87},\cite{mura87},\cite{thistl87}, 
and by the proof of the
Tait Flyping Conjecture \cite{flype} the diagram $D^\prime$ is related to the
reference diagram $R$ of the same alternating knot by a sequence of flypes
(Figure~\ref{fig:flype1}). 

\begin{figure}[H]
\begin{center}
\includegraphics[width=3.5in]{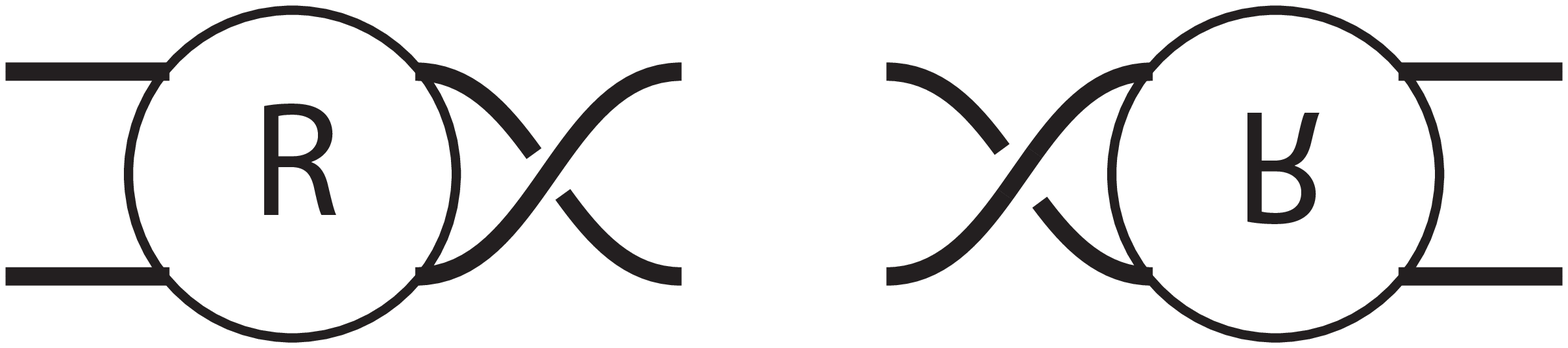}
\caption{Flype}\label{fig:flype1}
\end{center}
\end{figure} 

By applying Lemma~\ref{lemma:flype} below (with $D$, $\widetilde D$, and $D_1$ in that Lemma
playing the roles of $D^\prime$, $R$, and $D$, respectively) a finite number of times, 
there is a diagram $\widetilde D$
obtained by changing crossings in $R$ such that $D$ and $\widetilde D$ are related
by a sequence of flypes.  Then $\widetilde D$ is represented by a
DT code in the set $\mathcal{D}$.  Moreover, since the flype move does
not change the knot represented by the diagram, the knots represented by
$D$ and $\widetilde D$ are the same, and so $\widetilde D \in \mathcal{D}_K$ is another
minimal diagram for the knot $K$.

Applying Lemma~\ref{lemma:flype} another time,
there is a diagram $\widetilde D_1$ obtained by changing a single
crossing in $\widetilde D_1$ 
such that $D_1$ and $\widetilde D_1$ are related by a sequence of flypes.
Then $\widetilde D_1$ is another projection of the knot $K_1$,
and since $\widetilde D_1$ is obtained by a single crossing change from
a diagram represented in $\mathcal{D}_K$, then $K_1$ is in $\bjref_K$.
\end{proof}

\begin{lemma} \label{lemma:flype}
If knot diagrams $D$ and $\widetilde D$ are related by a sequence of flypes, then 
for any diagram $D_1$ obtained by changing a crossing in $D$ there is a 
diagram $\widetilde D_1$ obtained by changing a crossing in $\widetilde D_1$ 
such that $D_1$ and $\widetilde D_1$ are related by a sequence of flypes.
\end{lemma}

\begin{proof} This can be established
by induction on the number of flypes used, by showing that a crossing change followed by a flype
yields the same diagram as the same flype followed by some crossing change;
therefore several flypes followed by a crossing change is the same as some
crossing change followed 
by several flypes.
There are three cases to consider. (1) If the crossing change is outside of the
tangle being flyped, then the same crossing change applied after the
flype will yield the diagram $\widetilde D_1$ related to $D_1$ by the same flype. (2) If the crossing change is inside of the
tangle being flyped, then the crossing is carried to a crossing in the new
diagram by the flype, and changing that corresponding crossing after the
flyped yields the same diagram as crossing change followed by flype. 
(3) If the crossing
change is the one which is removed by the flype, 
then applying the crossing change followed by the flype
yields the same diagram obtained by first applying
the opposite flype and 
then changing the crossing created by the flype (illustrated in Figure~\ref{fig:flype2}).  
\begin{figure}[H]
\begin{center}
\includegraphics[width=3.5in]{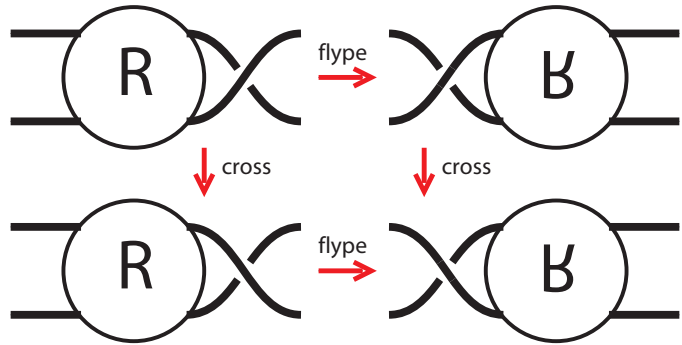}
\caption{Flype and crossing change}\label{fig:flype2}
\end{center}
\end{figure}
\end{proof}

Thus we lose no information about the knots that we need to consider in 
computing the weak Bernhard-Jablan unknotting number if we only work with 
the diagrams $\widetilde D$ for $K$ obtained by changing crossings in our reference diagrams.


\subsection{Proof of Theorem~\ref{thm:example}} \label{sub:example}


The remainder of the proof of Theorem~\ref{thm:fail} in this
section is the proof of parts (a)-(c) of Theorem~\ref{thm:example}. 
We begin with part (a) of that theorem.

\begin{lemma}
The knot $K13n3370$ has unknotting number less than or equal to 2.
\end{lemma}

\begin{proof} The knot $K13n3370$ is the closure of the 7-braid \hfill
$$\{1,1,-3,4,-3,-5,5,-6,-6,4,-5,2,4,-6,3,4,-1,3,5,2\},$$
with 20 crossings, using the notational convention of Knotinfo, 
shown in Figure~\ref{fig:braid}.  

\begin{figure}[H]
\begin{center}
\includegraphics[width=4.5in]{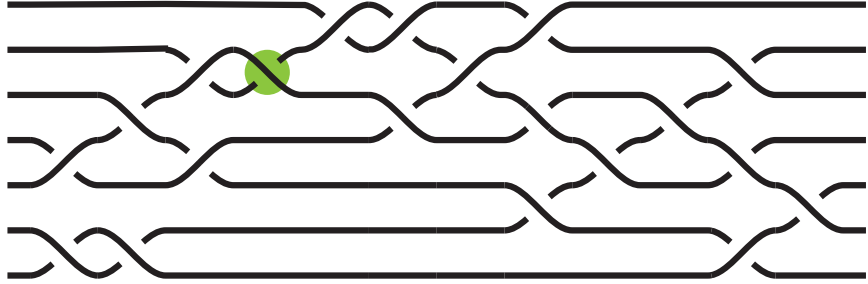}
\caption{$K13n3370$}\label{fig:braid}
\end{center}
\end{figure} 

SnapPy can convert this braid representation
to a Dowker-Thistlethwaite (DT) code for $K13n3370$, 
using the command {\it DT\_code()}, which SnapPy handles more natively 
than braid notation,
giving us the DT code 
$$
[-14,-10,-22,20,34,-36,40,-24,-8,6,-18,-32,2,38,26,-12,-16,4,-30,28].
$$
SnapPy verifies that this code
represents the knot $K13n3370$.
Changing the seventh crossing in the code gives us the 
DT code
$$
[-14,-10,-22,20,34,-36,-40,-24,-8,6,-18,-32,2,38,26,-12,-16,4,-30,28]
$$
which SnapPy identifies as the knot $K11n21$. From the 
Knotinfo database, we find that $u(K11n21)=1$; in fact, it is the case that all 23 
minimal crossing diagrams for the knot $K11n21$ have at least 
one crossing which yields the unknot when changed. DT codes for 
these 23 diagrams can be found in an appendix at the end of this paper.

This shows that a single crossing change to the given diagram 
for $K13n3370$ yields a knot with unknotting number $1$; therefore,
$K13n3370$ has unknotting number at most $2$.
\end{proof}

Next, we show the other result in part~(a) of Theorem~\ref{thm:example},
that the weak Bernhard-Jablan unknotting number 
$u_{BJ}^w(K13n3370)$ is equal to 3, as well as parts~(b-c) of that Theorem.
In overview, we use the algorithm of Section~\ref{sub:bjset} to 
find the sets $\mathcal{D}_K$
and $\bjref_K$ associated to the knot $K13n3370$, 
and iteratively find the sets $\mathcal{D}_{K'}$ and $\bjref_{K'}$
for all $K' \in \bjref_K$, etc.  We then use the definition of $u_{BJ}^w$
(or iteratively
apply Lemma~\ref{lem:bjnumbers}) and 	
Proposition~\ref{prop:refsuffices}, as well as other information on
unknotting numbers for knots in the set $\bjset_{K13n3370}$,
to determine the
weak Bernhard-Jablan unknotting number of the knot $K13n3370$.

In more detail, working directly with SnapPy (or equivalently, with Knotscape)
to build these sets, we find that there are 
$4878$ 13-crossing alternating knots, which yield $4878$ reference diagrams, 
and changing every subset of 12 of the crossings in each diagram
yields $2^{12}$ diagrams for each, for a total of $19,980,288$ diagrams for SnapPy to 
identify. Distributing the work over several machines made this a manageable 
task. The resulting data can be found at the website described on page~\pageref{url} of this paper. 
Among these diagrams there are $64,399$ diagrams which represent
13-crossing knots (also to be found at the website; for comparison, 
$6,122,841$ of the diagrams represent the unknot), and 24 minimal diagrams for the knot
$K13n3370$; their set $\mathcal{D}_{K13n3370}$ of DT codes 
are listed in the appendix.

Changing each single crossing in each of these 24 diagrams yields 312 diagrams, 
but only 13 distinct knots, according to 
SnapPy: they are the knots 
in the set 
\begin{align*}
\bjset_{K13n3370}=\{7_4, & 8_8, 10_{34}, K11a211, K11n91, K11n132, 
K12a1118, K12n288, \\
& K12n333, K12n469, K12n491, K12n501, K12n512\}.
\end{align*}

For at least 10 of these 13 knots, the unknotting number, and hence 
(by Lemma~\ref{lem:bjnumbers}(2)) also the weak Bernhard-Jablan unknotting number,
is at least 2.

\begin{lemma} \label{lemma:cases}
The knots $7_4$, $8_8$, $10_{34}$, $K11a211$, $K11n91$, $K11n132$, 
$K12a1118$, $K12n333$, $K12n469$, and $K12n512$ all have unknotting number at least 2.
\end{lemma}

\begin{proof} Lickorish \cite{lick82} showed that $u(7_4)=2$, 
Kanenobu and Murakami \cite{kamu86} showed that $u(8_8)=u(10_{34})=2$, and
Lewark and McCoy \cite{lemc15}, using the smooth 4-ball genus $g_4$, 
showed that $u(K11a211)\geq g_4(K11a211)=2$. Borodzik and Friedl \cite{bofr15}, using the
algebraic unknotting number $u_a$, showed that $u(K)\geq u_a(K)=2$ for 
$K=K11n91,K11n132,K12a1118,K12n333$, and $K12n469$. The authors showed (see \cite{knotinfo}), 
using the same random search procedure used to discover the projection
of $K13n3370$ used in this paper, that the knot $K12n512$ has a projection with a
crossing change yielding the knot $9_{38}$, and $u(9_{38})=3$ was shown by Owens \cite{ow08},
following earlier work of Stoimenow \cite{sto04}.
The relevant braids to see this knot adjacency are 
$$\{2,3,-4,-4,1,-2,3,-3,5,-3,-1,4,2,1,-5,-4,3,3,1,1,2\}$$
 for $K12n512$ and
$$\{2,3,-4,4,1,-2,3,-3,5,-3,-1,4,2,1,-5,-4,3,3,1,1,2\}$$
for $9_{38}$.
Therefore, all of these knots have unknotting number at least 2.
\end{proof}

This leaves $K12n288, K12n491$, and $K12n501$, whose unknotting numbers 
at the time of this writing are still unknown. But the same exhaustive 
procedure from Section~\ref{sub:bjset} 
applies to find the set $\mathcal{D}_K$ 
of minimal crossing diagrams for each 
$K \in \{K12n288, K12n491, K12n501\}$, as well as, 
incidentally, the minimal diagrams  (up to flypes) of all 12-crossing knots. 
The data set for all 12-crossing knots is available on the website 
discussed below (on page~\pageref{url}).
These three knots have, respectively, 24, 9, and 18 such diagrams,
which are given in the appendix.
SnapPy verifies that no single crossing change in any of these 51 diagrams 
yields the unknot, and so (using Proposition~\ref{prop:refsuffices}) 
the sets $\bjset_K$ for 
$K \in \{K12n288, K12n491, K12n501\}$ do not contain the unknot.
Hence these three 12-crossing knots all have
both strong and 
weak Bernhard-Jablan unknotting
number at least 2.  
Lemma~\ref{lem:bjnumbers}(1) now shows
that $u_{BJ}^w(K13n3370) \ge 3$.

To show part (b) of Theorem~\ref{thm:example}, it now
suffices to show the further claim that $u_{BJ}^w(K) \le 2$ for all
$K \in \{K12n288, K12n491, K12n501  \}$; we give an explicit
construction using SnapPy computations.
The knot $K12n288$ has a minimal crossing diagram with DT code
$$[4,\underline{10},12,-16,2,8,-18,-22,-6,-24,-14,-20];$$ 
when the second crossing is changed (changing the sign of the underlined 
entry in the DT code), SnapPy identifies the resulting knot as
the trefoil $3_1$.
The knot $K12n491$ has minimal diagram with DT code
$$[\underline{6},-12,20,18,24,16,-4,22,8,2,14,10],$$ 
and changing the first crossing yields the knot $6_3$.  Finally,
the knot   $K12n501$ has diagram with DT code
$$[6,-10,\underline{22},24,-16,18,20,-2,4,14,12,8],$$ 
and changing the third crossing results in the knot $8_{13}$.
It is a straightforward computation
to check that every $K' \in \{3_1,6_3,8_{13}\}$
has a minimal diagram admitting a single crossing change to 
the unknot, and hence satisfies $u_{BJ}^w(K')=1$.
(An alternative proof that each  $u_{BJ}^w(K')=1$ follows from
the fact that each $u(K')=1$ (see for example
KnotInfo~\cite{knotinfo}), and a result of McCoy~\cite{mc17}
that $u(K')=1$ implies $u_{BJ}^w(K')=1$ for all alternating knots.)
The claim now follows from Lemma~\ref{lem:bjnumbers}(1).

Applying Lemma~\ref{lem:bjnumbers}(1) again, to the knot
$K13n3370$, shows the last part of Theorem~\ref{thm:example}(a),
that $u_{BC}^w(K13n3370)=3$, as well as part (c) of Theorem~\ref{thm:example}.

As a consequence of these computations,
either one of the knots $K12n288, K12n491$, and $K12n501$
has unknotting number 1, and so fails to satisfy Equation~(\ref{eqn:bj})
in the Bernhard-Jablan Conjecture, or else all of the knots obtained from 
a minimal diagram of
$K13n3370$ by changing a crossing have unknotting number at least 2, and
so $K13n3370$ has strong Bernhard-Jablan unknotting number 3
and Equation~(\ref{eqn:bj}) fails for $K13n3370$.
This completes the proof of Theorems~\ref{thm:example} and~\ref{thm:fail},
of the existence of a counterexample to the Bernhard-Jablan Conjecture.


\section{Further examples}\label{sec:further}


In further computations we have found more counterexamples to
the Bernhard-Jablan Conjecture.  The following three pairs $(K,\tilde K)$
each satisfy that there is a diagram for $K$ with
a single crossing change resulting in a diagram for $\tilde K$, and
$u(\tilde K) +2 \le u_{BJ}^w(K)$.  Since $K$ and $\tilde K$ are crossing
adjacent, then
$u(K) \le u(K)+1$, and hence $u(K) < u_{BJ}^w(K)$, giving
a failure of Equation~(\ref{eqn:weakbj}) for the knot $K$.

Each example indicates a further counterexample to Equation~(\ref{eqn:bj})
in the Berhhard-Jablan Conjecture
either by the knot $K$, or by a knot in one of the the sets
$\bjset_{K}$, $\bjset^2_{K}=\cup_{ K' \in \bjset_{K}} \bjset_{ K'}$,
or $\bjset^i_{K}=\cup_{K' \in \bjset^{i-1}_{K}} \bjset_{ K'}$ for some $i \ge 3$.

The first pair is 
$K=K13n1669$, which has weak Bernhard-Jablan unknotting number $4$, and
$\tilde K=K14n23648$, with unknotting number $\leq$ $2$.
Adjacent braid representatives for these two knots are given by:
\begin{align*}
K = K13n1669: & \{-2, -4, -1, -3, -4, -4, -2, -3, 7, -1, -2, 4, -1, 5, 3, -1, -1, -2, 1, 6, 1, 4, 3\}\\
\tilde K = K14n23648: & \{-2, -4, 1, -3, -4, -4, -2, -3, 7, -1, -2, 4, -1, 5, 3, -1, -1, -2, 1, 6, 1, 4, 3\}
\end{align*}

The second example is
$K = K13n1587$, which has weak Bernhard-Jablan unknotting number $3$, and
$\tilde K = 10_{113}$, with unknotting number $1$.  Their adjacent
braid representations are:
\begin{align*}
K13n1587: &\{-4,-3,-3,-6,4,5,-4,3,-6,3,-2,1,3,2,5,4,1,1,6,-6,-5,3\}\\
10_{113}: &\{-4,-3,-3,-6,4,5,-4,3,-6,3,-2,1,3,2,5,4,1,1,6,6,-5,3\}
\end{align*}

Finally, the third pair gives the first example of an alternating
knot for which Equation~(\ref{eqn:weakbj}) fails.  The knots are
$K = K14a2539$, which has weak Bernhard-Jablan unknotting number $4$, and
$\tilde K = K14n1045$, with unknotting number $\leq$ $2$.
Adjacent braid representations are given by:
\begin{align*}
K14a2539:& \{7,3,-4,1,6,7,5,6,-1,5,-2,-7,-3, 4,-1,2,7,-6,3,7,\\
& 6,-2,-6,-1,-2,-7,3,-4,1,7,-3,7,-6,-7,5,1,-7,-1,6\}\\
K14n1045: &\{7,3,-4,1,6,7,5,6,-1,5,-2,-7,-3,-4,-1,2,7,-6,3,7,\\
& 6,-2,-6,-1,-2,-7,3,-4,1,7,-3,7,-6,-7,5,1,-7,-1,6\}
\end{align*}


\section{The future}\label{sec:fut}


With the failure of the Bernhard-Jablan Conjecture, there would seem to be no natural
remaining candidate for an algorithmic approach to computing unknotting number.
That is, there is no other `canonical' finite collection of diagrams for a knot in 
which to posit that an unknotting number minimizer `should' be found.
This could be viewed as a blow to the theory, but we choose to take a different
view. For example, despite a massive amount of effort, none of the computable invariants
that can be tied to unknotting number (and there are many) have succeeded
in showing that any of the nine 10-crossing knots with 
unknown unknotting number, namely $10_{11},10_{47},10_{51},10_{54},10_{61},
10_{76},10_{77},10_{79}$ and $10_{100}$,
have unknotting number 3, although all nine can be shown to
have weak Bernhard-Jablan unknotting number equal to 3. With the examples 
we have found, it may be more reasonable now to suppose that this
failure to establish that $u(K)=3$ for $K$ in this list is in fact because it
should not be possible; perhaps some or all of these knots have
unknotting number 2, and we just have not yet found the projection(s)
to demonstrate this. We are heartened by the fact that the example 
we found to settle Conjecture~\ref{conj:bj} 
was provided by a 20-crossing projection (although we
should note that the braid that our search program actually found 
had 37 crossings - it was possible, by hand, to reduce it to the 
braid word we have provided here), and so it may even be the case
that a relatively small projection could provide the knot adjacency
needed to show that one or more of these 
10-crossing knots have unknotting number 2.

There are almost certainly many more knots out there that
provide a counterexample to the Bernhard-Jablan conjecture. All of the code that
we have developed as part of this project to uncover unknotting numbers
by random search is available at the website \label{url}

\begin{center}
http://www.math.unl.edu/$\sim$mbrittenham2/unknottingsearch/
\end{center}

\noindent and we invite anyone who is interested in searching for interesting
knot adjacencies to download and run this code. The code to implement
random searches in SnapPy is written in python, as is the code for SnapPy
to build the databases of minimal crossing projections. SnapPy
can natively identify any knot with 14 or fewer crossings, and so 
this code is implemented to work up through that crossing limit.
The code to work with Knotscape provided at the website, which we used to check the
database computations, is written in Perl, and makes calls to the
Knotscape program `knotfind' to make its identifications. 

Since this code was originally designed to find useful knot adjacencies, which
would then identify tighter bounds on unknown unknotting numbers, and there
has been no systematic work, other than this project, to compute the unknotting numbers
of 13- and 14-crossing knots generally (that we are aware of), further
runs of this code are likely to uncover new values of the unknotting numbers of
these (as well as possibly smaller) knots. 
We would be grateful to hear of any such new data that the
reader might discover as a consequence of running this code for themselves;

\smallskip

The failure of Conjecture~\ref{conj:bj} in general leads naturally to several still unresolved questions.
Is Conjecture \ref{conj:bj} true for alternating knots? That is, does every minimal
\underbar{alternating} projection contain a crossing whose change will lower the
unknotting number of the underlying knot? McCoy \cite{mc17} has shown that this is 
true for unknotting number 1 alternating knots. (We note that
the third example in Section~\ref{sec:further} shows an alternating knot for which 
Equation~(\ref{eqn:weakbj}) fails to hold, but does not determine whether
Equation~(\ref{eqn:bj}) holds for that knot.)
In Theorem~\ref{thm:example} we have 
shown that one of four knots must provide a counterexample to 
Conjecture~\ref{conj:bj} (and in particular to Equation~(\ref{eqn:bj})), 
but the question remains, which one? 
What is the smallest (in terms of crossing number) counterexample
to the Bernhard-Jablan Conjecture? Are there infinitely many counterexamples? 

\medskip

\section{Appendix}

To make this paper as independent of external data sources as possible, 
we list, for reference, the DT codes of all of the minimal crossing
diagrams used to carry out the computations above. This will allow
the reader to verify for themselves that crossing changes on these diagrams have the
properties claimed in the paper. We do not provide the code here that was used to
produce these lists, or to show that they are exhaustive; that would make this
paper prohibitively long. This code is available at the website listed above.

\medskip

DT codes for the minimal diagrams, up to flype, of the knot $K11n21$:

\smallskip

{\smaller
\nidt $[4, 8, -12, 2, 16, -6, 20, 18, 10, 22, 14]$
\hfill
$[4, 8, 12, 2, -14, -18, 6, -20, -10, -22, -16]$

\nidt $[4, 8, 12, 2, -16, 18, 6, -20, 22, -14, -10]$
\hfill
$[4, 8, 12, 2, -20, 18, 6, -10, 22, -14, 16]$

\nidt $[4, 8, 12, 2, 14, -18, 6, -20, -10, -22, -16]$
\hfill
$[4, 8, 12, 2, 16, 18, 6, -20, 22, -14, 10]$

\nidt $[4, 8, 12, 2, 18, -16, 6, 20, -22, 14, -10]$
\hfill
$[4, 8, 12, 2, 18, -20, 6, -10, -22, 14, -16]$

\nidt $[4, 8, 12, 2, 18, 14, 6, -20, -22, 10, -16]$
\hfill
$[4, 8, 12, 2, 18, 20, 6, 10, -22, 14, -16]$

\nidt $[4, 10, -14, -20, 2, 16, -18, 8, -22, -6, -12]$
\hfill
$[4, 10, -14, -20, 2, 16, -18, 8, 22, -6, 12]$

\nidt $[4, 10, -14, -20, 2, 16, -22, 8, 12, -6, -18]$
\hfill
$[4, 10, -14, -20, 2, 16, -8, 22, 12, -6, -18]$

\nidt $[4, 10, -14, -20, 2, 22, -18, 8, -12, -6, 16]$
\hfill
$[4, 10, -14, 18, 2, 16, -6, 22, 20, 8, 12]$

\nidt $[4, 10, -16, -20, 2, 22, -18, -8, -12, -6, -14]$
\hfill
$[4, 10, 12, -14, 2, 22, -18, -20, -6, -8, -16]$

\nidt $[4, 10, 14, -20, 2, -16, -22, 8, -12, -6, -18]$
\hfill
$[4, 10, 14, -20, 2, -16, -8, 22, -12, -6, -18]$

\nidt $[4, 10, 16, -20, 2, 22, 18, -8, 12, -6, -14]$
\hfill
$[4, 12, -18, 14, 20, 2, 22, 8, 10, -6, 16]$

\nidt $[4, 14, 10, -20, 22, 18, 2, -8, 6, 12, -16]$
}

\medskip

DT codes for the minimal diagrams, up to flype, of $K13n3370$:

\smallskip

{\smaller
\nidt $[4,12,-22,-20,-18,-16,2,-24,-10,-8,-26,-14,-6]$

\nidt $[6,-10,12,14,-18,20,26,-24,-22,-4,2,-16,-8]$

\nidt $[6,-10,12,16,-18,20,26,24,-22,-4,2,-8,14]$

\nidt $[6,-10,12,22,-18,20,26,24,-8,-4,2,16,14]$

\nidt $[6,-10,12,24,-18,20,26,-8,-22,-4,2,-16,14]$

\nidt $[6,-10,12,26,-18,20,-8,-24,-22,-4,2,-16,-14]$

\nidt $[6,10,18,-16,4,24,22,20,-26,2,14,12,-8]$

\nidt $[6,-10,18,26,20,-2,24,22,8,4,16,14,12]$

\nidt $[6,-10,20,14,-18,-4,26,-24,-22,-12,2,-16,-8]$

\nidt $[6,-10,20,16,-18,-4,26,24,-22,-12,2,-8,14]$

\nidt $[6,-10,20,22,-18,-4,26,24,-8,-12,2,16,14]$

\nidt $[6,-10,20,24,-18,-4,26,-8,-22,-12,2,-16,14]$

\nidt $[6,-10,20,26,-18,-4,-8,-24,-22,-12,2,-16,-14]$

\nidt $[6,-10,22,-14,-2,-20,-8,26,24,-12,4,18,16]$

\nidt $[6,-12,22,-14,-2,20,-8,26,24,4,10,18,16]$

\nidt $[6,14,20,18,-24,-16,4,-22,-26,2,-12,-10,-8]$

\nidt $[6,14,20,26,18,-16,4,-22,-24,2,-12,-10,8]$

\nidt $[6,14,20,26,-24,-16,4,-22,-8,2,-12,-10,18]$

\nidt $[6,-14,22,20,18,26,-4,-24,10,8,2,-12,-16]$

\nidt $[6,-14,-22,26,20,18,-4,24,12,10,8,-2,16]$

\nidt $[6,20,12,14,-18,4,26,-24,-22,-10,2,-16,-8]$

\nidt $[6,20,12,16,-18,4,26,24,-22,-10,2,-8,14]$

\nidt $[6,20,12,22,-18,4,26,24,-8,-10,2,16,14]$

\nidt $[6,20,12,24,-18,4,26,-8,-22,-10,2,-16,14]$
}

\medskip

DT codes for the minimal diagrams, up to flype, of $K12n288$:

\smallskip

{\smaller
\nidt $[4,10,12,-16,2,8,-18,-22,-6,-24,-14,-20]$
\hfill
$[4,10,12,16,2,8,-18,-22,6,-24,-14,-20]$

\nidt $[4,10,12,-16,2,8,20,-6,24,22,14,18]$
\hfill
$[4,10,14,-12,2,20,-18,22,-6,8,24,16]$

\nidt $[4,10,14,-12,2,20,-18,22,6,-8,24,16]$
\hfill
$[4,10,-14,16,2,8,20,-24,22,12,18,-6]$

\nidt $[4,10,14,-16,2,8,20,-24,22,12,18,-6]$
\hfill
$[4,10,-14,18,2,8,20,-24,-22,12,-6,-16]$

\nidt $[4,10,14,-18,2,20,6,22,12,-8,24,16]$
\hfill
$[4,10,14,-20,2,8,-18,22,-6,-12,24,16]$

\nidt $[4,10,14,20,2,8,-18,22,6,-12,24,16]$
\hfill
$[4,10,14,-20,2,18,24,-22,8,12,-16,-6]$

\nidt $[4,10,14,-22,2,-18,8,-24,-20,-12,-16,-6]$
\hfill
$[4,10,-16,12,2,8,-20,-6,-24,-22,-14,-18]$

\nidt $[4,10,16,12,2,8,-20,6,-24,-22,-14,-18]$
\hfill
$[4,10,-16,14,2,8,-20,24,-22,-12,-18,-6]$

\nidt $[4,10,16,-14,2,8,-20,24,-22,-12,-18,-6]$
\hfill
$[4,10,16,-20,2,18,24,22,8,12,-6,14]$

\nidt $[4,10,-18,-12,2,16,-20,-8,24,-22,-14,-6]$
\hfill
$[4,10,18,-12,2,16,-20,8,24,-22,-14,-6]$

\nidt $[4,10,18,-14,2,8,-20,24,22,-12,-6,16]$
\hfill
$[4,10,18,16,2,8,-20,-22,24,-14,-12,6]$

\nidt $[4,10,18,-20,2,22,8,-6,24,-14,12,16]$
\hfill
$[4,10,20,-16,2,18,8,-22,12,24,-14,-6]$
}

\medskip

DT codes for the minimal diagrams, up to flype, of $K12n491$:

\smallskip

{\smaller
\nidt $[6,-12,20,18,24,16,-4,22,8,2,14,10]$
\hfill
$[6,-16,18,22,2,-4,24,20,-10,12,8,14]$

\nidt $[6,16,-18,22,2,-4,24,20,-10,12,8,14]$
\hfill
$[4,10,-16,-20,2,-18,-22,-8,-24,-14,-6,-12]$

\nidt $[6,-10,12,22,16,-18,24,20,-2,4,8,14]$
\hfill
$[6,-10,12,22,16,-18,24,20,2,-4,8,14]$

\nidt $[6,10,-18,22,2,16,24,20,-4,12,8,14]$
\hfill
$[6,-10,-18,22,16,-4,24,20,-2,12,8,14]$

\nidt $[6,-10,18,22,16,-4,24,20,2,12,8,14]$
}

\medskip

DT codes for the minimal diagrams, up to flype, of $K12n501$:

\smallskip

{\smaller
\nidt $[6,-10,22,24,-16,18,20,-2,4,14,12,8]$
\hfill
$[6,-10,22,24,-16,18,20,2,-4,14,12,8]$

\nidt $[6,-10,22,24,16,-18,-20,-4,2,-14,-12,8]$
\hfill
$[6,-10,22,24,16,-18,-20,4,-2,-14,-12,8]$

\nidt $[6,-12,20,24,18,16,-4,22,10,2,14,8]$
\hfill
$[4,10,-16,-22,2,-18,-20,-8,-24,-14,-12,-6]$

\nidt $[6,10,12,-18,4,2,-20,-22,-8,-24,-16,-14]$
\hfill
$[6,10,12,18,4,2,-20,-22,8,-24,-16,-14]$

\nidt $[6,10,12,-18,4,2,20,22,24,-8,16,14]$
\hfill
$[6,-10,12,24,16,-18,22,20,-2,4,14,8]$

\nidt $[6,-10,12,24,16,-18,22,20,2,-4,14,8]$
\hfill
$[6,-10,-14,24,16,-20,-18,22,-2,-12,-4,8]$

\nidt $[6,-10,14,24,16,-20,-18,22,2,-12,-4,8]$
\hfill
$[6,10,-16,24,2,-18,-20,-22,-4,-14,-12,8]$

\nidt $[6,10,16,24,2,-18,-20,-22,4,-14,-12,8]$
\hfill
$[6,-10,16,24,-18,-4,-20,-22,2,-14,-12,8]$

\nidt $[6,-10,-18,24,16,-4,20,22,-2,14,12,8]$
\hfill
$[6,-10,18,24,16,-4,20,22,2,14,12,8]$
}




\end{document}